\documentclass[a4paper,notitlepage,twoside,leqno,11pt]{amsart}
\usepackage{mathrsfs}
\usepackage{epic,eepic}
\usepackage{bbm,pifont,latexsym}
\usepackage{anysize}
\marginsize{2.4cm}{2.4cm}{2.4cm}{2.4cm}

\usepackage{dcolumn,indentfirst,color}
\usepackage[pdfpagemode=UseNone,pdfstartview=FitH]{hyperref}
\usepackage{amsmath,amssymb,amscd,amsthm,amsfonts,mathrsfs}
\usepackage{color,graphicx,xcolor,graphics}
\usepackage{tikz}
\usepackage{titlesec}
\usepackage{titletoc}
\titlecontents{section}[20pt]{\addvspace{2pt}\filright}
              {\contentspush{\thecontentslabel. }}
              {}{\titlerule*[8pt]{}\contentspage}

\usepackage{fancyhdr}
\newcommand{\wt}{\widetilde}

\newcommand{\mb}{\mathbb}
\newcommand{\mc}{\mathcal}
\newcommand{\sm}{\setminus}
\newcommand{\tu}{\textup}
\newcommand{\ol}{\overline}
\newcommand{\es}{\emptyset}
\newcommand{\mt}{\mapsto}

\newcommand{\tb}{\textbf}

\newcommand{\wh}{\widehat}
\newcommand{\olC}{\overline{\mathbb{C}}}

\newtheorem{thm}{Theorem}[section]
\newtheorem{lema}[thm]{Lemma}
\newtheorem{lem}[thm]{Lemma}

\newtheorem{prop}[thm]{Proposition}

\newtheorem{rmk}[thm]{Remark}
\newtheorem{defi}{Definition}[section]

\makeatletter\@addtoreset{equation}{section}\makeatother

\titleformat{\section}{\centering\normalsize}{\textsc{\thesection.}}{1em}{\textsc}
\titleformat{\subsection}{\normalsize}{\thesubsection.}{1em}{\textbf}

\title
 {The landing of parameter rays at non-recurrent critical portraits}

\author{Yan Gao}
\address{Yan Gao, School of Mathematics, Sichuan University, Chengdu 610064, China}
\email{gyan@scu.edu.cn}

\author{Jinsong Zeng}
\address{Jinsong Zeng, School of Mathematics and Information Science, Guangzhou University, Guangzhou 510006, China}
\email{jinsongzeng@163.com}

\subjclass[2010]{Primary: 37F45; Secondary: 37F10}

\keywords{critical portraits; non-recurrent; impressions;}

\begin{document}
\maketitle
\begin{abstract}
	Based on the distortion theory developed by Cui--Tan  \cite{CT15}, we prove the landing of every parameter ray at critical portraits coming from non-recurrent polynomials, thereby generalizing a result of  Kiwi \cite[Corollary]{Ki05}.
\end{abstract}
\section{Introduction}

In Complex Dynamicscomplex dynamics, a central topic is to study the dynamics of the quadratic family $\{f_c:z\mapsto z^2+c\}$, and the structure of the corresponding \emph{Mandelbrot set}
\[\mathcal{M}:=\{c\in\mathbb{C}\mid f_c^n(c)\not\to\infty\text{ as }n\to\infty\}.\]
A fundamental result in this aspect, due to Douady and Hubbard \cite{DH84}, is to describe the landing behavior of the rational \emph{parameter rays}, which are defined to be the images of $(1,\infty)e^{2\pi i\theta}$ under the Riemann mapping
\begin{equation}\label{eq:r-map}
\text{$\Phi:\mathbb{C}\setminus \mathcal{M}\to \mathbb{C}\setminus\ol{\mathbb{D}}$, \  with $\Phi(z)/z\to 1$ as $z\to\infty$}.
\end{equation}

In \cite{CT10}, the authors give an alternative proof of the Douady--Hubbard ray-landing theorem for quadratic Misiurewicz polynomials by using a new type of distortion theorem given in \cite{CT15}. Based on this distortion theorem and with a similar idea mentioned in \cite{CT10} and \cite[Section 8]{CT15}, we can generalize the ray-landing theorem to \emph{non-recurrent} polynomials with any degree $d\geq 2$.

Precisely,
let $\mc{P}_d\cong \mb{C}^{d-1}$ denote the parameter space of  monic, centered polynomials of degree $d$, i.e., $f(z)=z^d+a_{d-2}z^{d-2}+\cdots+a_0$. The \emph{connectedness locus} $\mc{C}_d$, comprising polynomials in $\mathcal{P}_d$ with all critical orbits bounded, is the generalization of the Mandelbrot set in higher degree case. In contrast, the \emph{shift locus} $\mc{S}_d$, formed by polynomials in $\mathcal{P}_d$ with all critical points escaping to infinity , is an analogy in the higher-degree case to the outside of the Mandelbrot set.

For $d\geq3$, a combinatorial quantity called \emph{critical portrait} is used to label the polynomials in place of \emph{character angles} in the quadratic case. Generally, a critical portrait associated with a polynomial is a collection of external angles corresponding to the critical points (see Section~\ref{sec:critical-portrait} for details). Let $\Theta$ be a critical portrait of degree $d$ and $r>0$ be any number. As proved in \cite[Theorem 3.7]{Ki05}, there exists a unique polynomial $f_r(\Theta)\in \mathcal{S}_d$ such that $\Theta$ is the critical portrait of $ f_r(\Theta)$ and the \emph{escaping rate} (see Section~\ref{sec:critical-portrait}) of each critical value is $r$, and that the map
\[R_{\mathcal{C}_d}(\Theta):(0,\infty)\to \mathcal{S}_d,\quad r\mapsto f_r(\Theta)\]
is injective. Thus, we call $R_{\mathcal{C}_d}(\Theta)$ the \emph{parameter ray} at the critical portrait $\Theta$. In case of $d=2$, a critical portrait is written as $\Theta=\{\theta/2,(\theta+1)/2\}$ for some $\theta\in \mathbb{R}/\mathbb{Z}$, and the parameter rays given above coincide with the ones defined by the Riemann mapping $\Phi$ in \eqref{eq:r-map}, namely
\[R_{\mathcal{C}_2}(\Theta)=R_{\mathcal{M}}(\theta):=\Phi^{-1}((1,\infty)e^{2\pi i\theta}).\]

In the quadratic case, Douady and Hubbard \cite{DH84} proved that every parameter ray at rational angles lands on $\mathcal{M}$. In the higher-degree case ($d\geq 3$), Kiwi\cite{Ki05} proved that parameter rays at strictly pre-periodic critical portraits land at a Misiurewicz polynomial. However, the landing of a parameter ray at a rational critical portrait $\Theta$ containing periodic angles remains unknown. In the present work, we show that every parameter ray at critical portraits coming from \emph{non-recurrent} polynomials will land. Our approach is similar to that of Cui--Tan \cite{CT10} but differs from that of Kiwi \cite{Ki05}.

A critical point $c$ of a rational map is said to be \emph{recurrent} if the orbit of $c$ is infinite and $c$ is an accumulation point of its orbit. A polynomial is called \emph{non-recurrent} if it has a connected Julia set but no bounded Fatou components and no recurrent critical points. For instance, a Misiurewicz polynomial is non-recurrent. Our main result is as follows.
\begin{thm}\label{main_thm}
	Let $\Theta$ be a critical portrait of a non-recurrent polynomial $f$ of degree $d\geq2$. Then the parameter ray $R_{\mathcal{C}_d}(\Theta)$ lands at $f$.
\end{thm}

\vskip 0.5cm
\tb{Notations.} We list few notations, herein, that are used throughout this paper.

- Let $\mathbb{C}$, $\olC$, and $\mathbb{C}^*$ denote the complex plane, the Riemann sphere, and the puncture plane, respectively.

- Let $\mathbb{D}:=\{z\in\mathbb{C}:|z|<1\}$, $\mathbb{D}(R):=\{z\in\mathbb{C}:|z|<R\}$ and $\mathbb{T}:=\mathbb{R}/\mathbb{Z}$. By abuse using of notations, we identify the unit circle with $\mathbb{T}$ by identifying a point with its argument.

Let $z\in\olC$, $E\subseteq \olC$, and $r>0$. We denote $B(z,r):=\{x\in\olC; \tu{dist}(x,z)<r\}$ and $\tu{diam}\, E:=\tu{sup}\{\tu{dist}\,(x,y);x,y\in E\}$, where $\tu{dist}\,(\,,\,)$ is the spherical metric. Furthermore, $B_e(z,r),{\rm diam}_e$ and ${\rm dist}_e(,)$ refer to the corresponding quantities in the Euclidean metric.

- In the present paper, with no special emphasis, the convergence of maps or points on $\olC$ or on a subset of $\olC$ is under the spherical metric.

- Let $S$ be a subset of $\olC$ and $f$ be a rational map. Then, $\tu{Orb}(S):=\cup_{k\geq 0}f^k(S)$. Note that here we require $k\geq 0$.

- Let $S$ be an open subset of $\olC$. Then, the collection of all components of $S$ is denoted by $\tu{Comp}(S)$.

- The Julia set and the Fatou set of a rational map $f$ are denoted by $\mc{J}_f$ and $\mc{F}_f$, respectively.

\section{Preliminaries}

\subsection{Modulus of an annulus}

Let $A\subseteq \olC$ be an annulus, such that each of its two complementary components contain at
least two points. Then, there exists a constant $r>1$ and a conformal map $\chi_A:A\to A(r)$,
where $A(r):=\{z:1<|z|<r\}$,
such that $r$ is unique and $\chi_A$ is unique up to post-composition of a rotation. The \emph{modulus} of $A$ is defined as
\[{\rm mod}\,A:=(\log r)/2\pi.\]

\subsection{Various distortions}
Let $U$ be a domain in $\olC$ and $z\in U$. The $\emph{Shape}$ of $U$ about $z$ is defined as
$$\tu{Shape}(U,z)=\frac{\tu{max}_{w\in\partial U}\tu{dist}(w,z)}{\tu{min}_{w\in\partial U}\tu{dist}(w,z)}.$$
It is obvious that $B(z,r)\subseteq U\subseteq B(z,kr)$ for some $r$, where $k:=\tu{Shape}(U,z)$. Thus, $U$ is a round disk centered at $z$ if and only whenif $\tu{Shape}(U,z)=1$.

In fact, the area of $U$ can be estimated by its Shape and diameter. To check this, by pre-composition of a rotation, one can assume $U\subseteq \mathbb{C}$, (any rotation retains the spherical metric) such that the Euclidean metric and the spherical metric are compatible in $U$. Thus, there is a constant $M$, depending only on $U$, such that
%\[\frac{1}{M}{\rm dist}_e(w,z)\leq {\rm dist}(w,z)\leq M \cdot {\rm dist}_e(w,z), \forall z,w\in U.\]
%It follows that
\[\frac{1}{M^2}\leq \frac{{\rm Shape}(U,z)}{{\rm Shape}_e(U,z)}\leq M^2 {\ \rm and\ } \frac{1}{M}\leq\frac{{\rm diam}(U)}{{\rm diam}_e(U)}\leq M.\]
A simple calculation shows that
\begin{equation}\label{eq:1}
\frac{\pi}{4}\frac{1}{M^4}\frac{{\rm diam}^2(U)}{{\rm Shape}^2(U,z)}\leq{\rm Area}_e(U)\leq\frac{\pi}{4} M^4{\rm Shape}^2(U,z){\rm diam}^2(U)
\end{equation}
and
\begin{equation}\label{eq:2}
\frac{\pi}{4}\frac{1}{M^6}\frac{{\rm diam}^2(U)}{{\rm Shape}^2(U,z)}\leq{\rm Area}(U)\leq\frac{\pi}{4} M^6{\rm Shape}^2(U,z){\rm diam}^2(U).
\end{equation}
%Let $K$ be a compact set in $\olC$ and $z_1\neq z_2\in K$, the \emph{turning} is defined as
%$$\Lambda(K,z_1,z_2)=\frac{\tu{diam}\, K}{\tu{dist}\,(z_1,z_2)}.$$
\begin{lema}\label{lema_area}
	Let $W$ and $W'$ be two nested Jordan domains in $\olC$ such that $\ol{W'}\subseteq W$. Let $m:={\rm mod} (W\setminus \ol{W'})$ and $C>1$ be a constant. Then,
	there exists $\lambda\in(0,1)$, depending only on $W$, $m$, and $C$, such that for any univalent mapping $h:W\to \mathbb{C}^*$ and any two nested open disks $E\subseteq U$ in $W'$ with
	$$\tu{Shape}(E,x)\leq C\tu{ for some $x\in E$\ \ and }\tu{diam}\,U\leq C\cdot\tu{diam}\,E,$$
	the following inequality
	$$\tu{Area}\,(\rho_*,h(E))\geq \lambda\, \tu{Area}\,(\rho_*,h(U))$$
	holds, where $\tu{Area}(\rho_*,S):=\iint_S\frac{1}{4\pi^2|w|^2}d\zeta d\eta$ with $\rho_*(w):=\frac{1}{2\pi|w|}$ is a planar metric on $\mathbb{C}^*$.
\end{lema}
\begin{proof}
	By pre-composing a M\"{o}bius transformation, we may assume that $W$ is bounded in $\mathbb{C}$. By inequality \eqref{eq:1} and the properties of the lemma, there exists a constant $C_1$ depending only on $C$ and $W$ such that
	$$\tu{Area}_e\,U\leq C_1\,\tu{Area}_e\,E.$$
	
	Applying the Koebe distortion theorem to $h:(W',W)\to(h(W'),h(W))$, we obtain a constant $C_2\geq 1$, depending only on $m$, such that
	$$\tu{max}_{z\in W'}|h'(z)|\leq C_2\,\tu{min}_{z\in W'}|h'(z)|.$$
	
	Because a univalent map preserves the modulus of an annulus, we have $\tu{mod}\,h(W)\setminus\ol{(h(W'))}=m$. Therefore,
	there exists a constant $C_3> 0$, depending only on $m$, such that
	$$\tu{diam}_eh(W')\leq C_3\,\tu{dist}_e(\partial h(W),\partial h(W')).$$
	Let $z_h$ be in $\partial W'$, such that $h(z_h):=\tu{max}_{z\in \partial W'}\tu{dist}_e(h(z),0)$. Then, for any $z\in W'$, we have
	$$\frac{|h(z_h)|}{|h(z)|}\leq \frac{{|h(z)|+\tu{diam}_eh(W')}}{|h(z)|}\leq 1+\frac{\tu{diam}_eh(W')}{\tu{dist}_e(\partial h(W),\partial h(W')}\leq 1+C_3.$$
	Combining the inequalities above, we estimate
	\begin{align}
	\tu{Area}(\rho_*,h(E))&=\iint_{h(E)}\frac{{1}}{4\pi^2|w|^2}d\zeta d\eta=\iint_E\frac{|h'(z)|^2}{4\pi^2|h(z)|^2}dxdy\notag\\
	&\geq \frac{1}{C_2^2}\iint_E\frac{|h'(z_h)|^2}{4\pi^2|h(z_h)|^2}dxdy\geq\frac{1}{C_1C_2^2}\iint_U\frac{|h'(z_h)|^2}{4\pi^2|h(z_h)|^2}dxdy\notag\\
	&\geq \frac{1}{C_1C_2^4(1+C_3)^2}\iint_U\frac{|h'(z)|^2}{4\pi^2|h(z)|^2}dxdy\notag\\
	&=\lambda\,\tu{Area}\,(\rho_*,h(U)).\notag
	\end{align}
	Thus, the lemma is proved.
\end{proof}

The following is a distortion theorem for holomorphic branched covering. One can refer to \cite[Lemma~6.1]{QWY12} for the proof of parts~1 and 2  and to \cite[Lemma~4.5]{KL09} for the proof of part~3.

\begin{lem}\label{controlling_holomorphic_distortion}
	Let $E_i\subseteq U_i\subseteq V_i\subseteq \olC$, $i=1,2$, be a pair of Jordan disks with $\tu{mod}(V_2\sm \ol{U_2})\geq m>0$. Suppose that $g:V_1\to V_2$ is a proper holomorphic map of degree $\leq d$, $U_1$ is a component of $g^{-1}(U_2)$, and $E_1$ is a component of $g^{-1}(E_2)$. Then, there exists a constant $C=C(d,m)$, depending only on $d$ and $m$, such that
	
	\begin{enumerate}
		\item for all $z\in E_1$, the shape satisfies
		$$\tu{Shape}(E_1,z)\leq C\,\tu{Shape}\,(E_2,g(z));$$
		\item ${\rm diam} U_1/{\rm diam} E_1\leq C\,{\rm diam} U_2/{\rm diam} E_2;$
		\item $\tu{mod }(V_1\sm\ol{U}_1)\leq \tu{mod }(V_2\sm\ol{U}_2)\leq d\,\tu{mod }(V_1\sm\ol{U}_1).$
	\end{enumerate}
\end{lem}

Following \cite{CT15}, the modulus difference distortion of an annulus is used to estimate the distortion between a univalent map and a M\"{o}bius transformation.
Let $V\subseteq \olC$ be an open set and $\phi:V\to \olC$ be a univalent map. Define
\[\mathfrak{D}(\phi,V):=\sup_{E_1,E_2\subseteq V}|{\rm mod}A(E_1,E_2)-{\rm mod}A(\phi(E_1),\phi(E_2))|,\]
where $E_1$ and $E_2$ are disjoint full continua in $V$ and $A(E_1,E_2):=\olC\setminus(E_1\cup E_2)$. The following result is an equivalent variation of \cite[Theorem 8.1]{CT15}.

\begin{lem}\label{lem:distorsion}
	Let $V\subseteq \olC$ be an open set containing three distinct points $z_1$, $z_2$, and $z_3$ and a disk $B(z_1,r_0)$ centered at $z_1$, and let $\phi:V\to \olC$ be a univalent map fixed at $z_1$, $z_2$, and $z_3$. Suppose that $\mathfrak{D}(\phi,V)=\delta<\infty$. Then, there exists a constant $C>0$, depending on only $z_1$, $z_2$, $z_3$, and $r_0$, such that $${\rm dist}(\phi(z),z)\leq C\cdot \delta.$$
\end{lem}

\subsection{Nested disk systems}\label{sec:system}
Our definition of $m$-nested, $\lambda$-scattered disk systems is a minor generalization of the concept of $s(r)$-nested, $\lambda$-scattered disk systems given in \cite[Section 8.2]{CT15}.

\begin{defi}
	Let $X$ be a finite set in $\olC$. A collection of open topological disks $\{D_x\}_{x\in X}$ is called a \emph{nested disk system} if
	\begin{itemize}
		\item each $D_x$ contains $x$;
		\item if $D_x\cap D_y\neq\es$ with $x\not=y$, then either $D_x\subsetneqq D_y$ or $D_x\supsetneqq D_y$.
	\end{itemize}
	
	Let $\{D_x\}_{x\in X}$ be a nested disk system. Let $\{D_x''\}_{x\in X}$ and $\{D_x''\}_{x\in X}$ be two sequences of topological disks. Then, $\{(D_x'',D_x',D_x)\}_{x\in X}$ is said to be $m$-$nested$ if
	\begin{itemize}
		\item  $x\in D''_x\subseteq D'_x\subseteq D_x$;
		\item  if $D_y\subseteq D_x$ and $D_y\cap D_x''\not=\emptyset$, then $D_y\subseteq D_x'$;
		\item $\tu{min}_{x\in X}\{ \tu{mod }D_x\sm\ol{D_x'}\}\geq m.$
	\end{itemize}
\end{defi}

Comparing with the original concept of $s(r)$-nested disk systems given in \cite{CT15}, we remark that if $\{D_x\}_{x\in X}$ is an $s(r)$-nested disk system, then $D_y$ must avoid $x$ if $D_y\subseteq D_x$. However, our definition of $m$-nested disk systems does not have this restriction.

Let $W\subseteq \olC$ be an open set and $\{D_x\}_{x\in X}$ be a nested disk system, such that $\ol{\cup_{x\in X} D_x}\subseteq W$. Let $V_x$ be the union of all $D_y(\neq D_x)$ contained in $D_x$. Then, the nested disk system $\{D_x\}_{x\in X}$ is said to be $\lambda$-$scattered$ in $W$ with $\lambda\in(0,1)$ for each $x\in X$ and any univalent map $h:W\to\mathbb{C}^{*}$,
$$\tu{Area}(\rho_*,h(V_x))\leq\lambda\,\tu{Area}(\rho_*,h(D_x)),$$
where $\tu{Area}(\rho_*,S):=\iint_S\frac{1}{4\pi^2|w|^2}d\zeta d\eta$ with $\rho_*(w):=\frac{1}{2\pi|w|}$ as a planar metric on $\mathbb{C}^*$.

The following result is a generalization of Theorem 8.4 in \cite{CT15}. The original theorem deals with $s(r)$-nested, $\lambda$-scattered disk systems whereas ours pertains to $m$-nested, $\lambda$-scattered disk systems, but the arguments are basically the same. One can refer to \cite[Theorem 8.4]{CT15} or \cite[Theorem 4.4]{Zeng15} for its proof.

\begin{prop}\label{pro:control}
	Let $W\subseteq\olC$ be an open set with $\olC\setminus \ol{W}\not=\emptyset$. Then, for any
	$m$-nested, $\lambda$-scattered disk system $\{(D_x'',D_x',D_x)\}_{x\in X}$ in $W$ with $m$ sufficiently large and for
	any univalent mapping $\phi:\olC\setminus\ol{\cup_{x\in X}D_x''}\to \olC$,
	there exists $C(m,\lambda)>0$, depending on only $m$, $\lambda$, and $W$ and with $C(m,\lambda)\to 0$ as $m\to\infty$, such that
	$$ \mathfrak{D}(\phi,\olC\setminus \ol{W})\leq C(m,\lambda).$$
\end{prop}

Combining Lemma \ref{lem:distorsion} and Proposition \ref{pro:control}, we get a distortion theorem of univalent maps off nested disk systems.

\begin{thm}\label{thm_nest}
	Let $W\subseteq \olC$ be an open set and $V$ be a Jordan domain with $V\subseteq \olC\setminus \ol{W}$. Let $z_1$, $z_2$, and $z_3$ be three distinct points in $V$. Then, for any
	$m$-nested, $\lambda$-scattered disk system $\{(D_x'',D_x',D_x)\}_{x\in X}$ in $W$ with $m$ sufficiently large and
	any univalent mapping $\phi:\olC\setminus\ol{\cup_{x\in X}D_x''}\to \olC$ fixing $z_1$, $z_2$, and $z_3$,
	there exists $C(m,\lambda)>0$, depending only on $m$, $\lambda$, $W$, and $V$ and with $C(m,\lambda)\to 0$ as $m\to\infty$, such that
	\[\sup_{z\in V}\{{\rm dist}(\phi(z),z)\}\leq C(m,\lambda).\]
\end{thm}

\subsection{Rational maps}
A rational map $f$ can be regarded as an analytic finite-branched covering of the Riemann sphere $\olC$. The set of \emph{critical points} corresponds to the points $c\in \olC$ such that the local degree $\tu{deg}(f,c)\geq 2$, that is, $f'(c)=0$ (in a suitable chart). A point $z\in\olC$ of period $p$ is called \emph{parabolic} (resp. \emph{repelling}) if the multiplier $(f^p)'(z)$ is a root of unity (resp. \,has modulus greater than $1$). Parabolic and repelling periodic points are contained in the Julia set. The orbit of a point $z$, denoted by ${\rm Orb}(z)$, represents the set $\{z,f(z),f^2(z),\ldots,\}$. The \emph{$\omega$-limit set} $\omega(z)$ of a point $z$ comprises the accumulation points of ${\rm Orb}(z)$. A point $z$ is said to be \emph{recurrent} if $z\in \omega(z)$.
The following result due to Ma\~{n}\'{e} will be used repeatedly; see \cite[Theorem II]{Ma93} or \cite{ST00}.

\begin{lem}[Ma\~{n}\'{e}'s lemma]\label{mane_lemma}
	Let $f:\olC\to\olC$ be a rational map with degree at least two. If a point $x\in \mc{J}_f$ is not a parabolic periodic point and is not contained in the $\omega$-limit set of a recurrent critical point, then for any $\epsilon>0$ there exist $\delta=\delta(x,\epsilon)<\epsilon$ and integer $\eta=\eta(x,\epsilon)$ such that for any $n\geq 0$ and any component $B_n$ of $f^{-n}B(x,\delta)$,
	\begin{enumerate}
		\item the diameter of $B_n$ is less than $\epsilon$;
		\item the degree of $f^n:B_n\to B(x,\delta)$ is less than $\eta$;
		\item $B_n$ is a topological disk;
		\item $\tu{diam}\,B_n\to 0$ as $n\to\infty$.
	\end{enumerate}
\end{lem}

\begin{defi}
	Let $f$ be a rational map. A topological disk $B\subseteq \olC$ is said to be \emph{$(\epsilon,\eta)$-backward stable}  if, for any $n\geq0$ and any component $B_n$ of $f^{-n}(B)$, properties~1--4 in Lemma~\ref{mane_lemma} hold when replacing $B(x,\delta)$ with $B$ in property~2.
\end{defi}

The convergence theorem below for a rational map sequence comes from \cite[Lemma~2.8]{CT15}.
\begin{lem}\label{convergence_of_rational_maps}
	Let $\{f_n\}$ be a sequence of rational maps with constant degree $d\geq 1$. Suppose that on a non-empty open subset $U\subset \olC$ the sequence $\{f_n\}$ converges uniformly to a map $g$ as $n\to\infty$.
	Then, $g$ is a rational map with $\tu{deg }g\leq d$. Moreover, if the equality   holds then $\{f_n\}$ converges uniformly to $g$ on the whole sphere $\olC$ as $n\to\infty$.
\end{lem}

\subsection{Critical portraits of polynomials}\label{sec:critical-portrait}

%The aim of this section is to fix notation and recall some results from
%polynomial dynamics. We refer the reader to [7] and [24] for background material
%in complex dynamics

Let $f$ be a monic centered polynomial of degree $d$. Then, there exists a conformal map $\psi_f$ near $\infty$, normalized by $\psi_f(z)/z\to 1$ as $z\to\infty$ and called the \emph{B\"{o}ttcher coordinate} of $f$, conjugating $f$ with $z\mapsto z^d$ near $\infty$. Its inverse $\psi_f^{-1}$ has a continuation to $\mathbb{C}\setminus \ol{\mathbb{D}(R)}$ provided that $\psi^{-1}_f(\mathbb{C}\setminus \ol{\mathbb{D}(R)})$ contains no critical points of $f$ except $\infty$.

The \emph{Green function} $G_f:\mathbb{C}\to [0,+\infty)$ defined by
\[G_f(z):=\lim_{n\to\infty}\log^+|f^n(z)|/d^n,\]
where $\log^+t:=\max\{0,\log t\}$, measures the rate at which points escape to $\infty$. It is a well-defined continuous function that vanishes on the \emph{filled-in Julia set} $\mathcal{K}_f$, i.e., the points $z\in\mathbb{C}$ such that $f^n(z)\not\to\infty$ as $n\to\infty$, and coincides with $\log|\psi_f(z)|$, where $\psi_f$ is well-defined. We say that a point $z$ has \emph{escaping rate} $r$ if $G_f(z)=r$.

If $\mathcal{K}_f$ is connected, the B\"{o}ttcher coordinate $\psi_f$ maps $\mathbb{C}\setminus \mathcal{K}_f$ onto $\mathbb{C}\setminus\ol{\mathbb{D}}$.
The preimage of $(1,\infty)e^{2\pi i\theta}$ under $\psi_f^{-1}$, denoted by $R_f(\theta)$, is called the \emph{external ray} of $f$ at angle $\theta$. We say that $R_f(\theta)$ \emph{lands} if $\ol{R_f(\theta)}\cap \mathcal{K}_f$ is a singleton .

If $\mathcal{K}_f$ is not connected, according to Levin and Sodin \cite{LS}, then for each $\theta\in\mathbb{T}$ there is an infimum $r_f(\theta)\geq 1$ such that $\psi_f^{-1}$ extends analytically along $(r_f(\theta),\infty)e^{2\pi i\theta}$. We denote by $R_f(\theta)$
the preimage of $(r_f(\theta),\infty)e^{2\pi i\theta}$ under $\psi^{-1}_f$, called the \emph{external ray} of $f$ at $\theta$. When $r_f(\theta)>1$,
as $s\searrow \log r_f(\theta)$, the point $\psi_f^{-1}(e^{s+2\pi i\theta})\in R_f(\theta)$ converges to a point $x$, which is an iterated preimage by $f$ of the critical points of $f$. In this case, we say that $R_f(\theta)$ \emph{bifurcates} at $x$.

Let $f$ be a non-recurrent polynomial, i.e., $f$ has a connected Julia set but no bounded Fatou components and no recurrent critical points. It is known that such $f$ has a locally-connected Julia set (see \cite{CJY94,Yin99}). One can then locate its critical points by the external rays landing on them. We begin by fixing $\ell$ external rays $R_f(\theta_1),\cdots, R_f(\theta_\ell)$ that land at the critical values $v_1,\cdots v_\ell$, respectively. Then, at each critical point $c_j$ with $f(c_j)=v_i$, let $\Theta(c_j)$ be the set formed by the arguments of $\tu{deg}(f,c_j)$ external rays landing at $c_j$ among the pullbacks of $R_f(\theta_i)$; the collection
\[\Theta(f):=\{\Theta(c_1),\ldots,\Theta(c_m)\}=:\{\Theta_1,\ldots,\Theta_m\}\]
is called a \emph{critical portrait} of $f$. Note that the critical portraits of $f$ are not necessarily unique but there are finitely many of them. It is also easy to check that $\Theta(f)$ satisfies the following properties.
\begin{description}
	\item[(CP1)] The map $m_d:\theta\mapsto d\cdot \theta\,(\tu{mod}\,\mathbb{Z})$ sends each $\Theta_j$ to a singleton.
	\item[(CP2)] Any two distinct $\Theta_i,\Theta_j$ are \emph{unlinked}, that is, $\Theta_i,\Theta_j$ can be totally covered by two disjoint internals in $\partial \mb{D}$.
	\item[(CP3)] $\sum_{1\leq j\leq m}(\#\Theta_j-1)=d-1$.
\end{description}

Purely combinatorial,  a finite collection of finite subsets $\Theta=\{\Theta_1,\ldots,\Theta_m\}$ of $\mathbb{T}$ is called a \emph{critical portrait} of degree $d$ if properties (CP1), (CP2), and (CP3) are satisfied.

For any $r>0$, we denote by $\mathcal{S}_d(r)$ the subset of $\mathcal{S}_d$ comprising polynomials with escaping rates of all critical values equal to $r$. 
Any polynomial $f$ in $\mathcal{S}_d(r)$ can also be associated with a critical portrait
$$\Theta(f):=\{\Theta(c_1),\ldots,\Theta(c_m)\}$$ such that each $\Theta(c_j)$ comprises the arguments of ${\rm deg}(f,c_j)$ external rays bifurcation  at $c_j$ which are mapped by $f$ to a single ray , and that $m_d(\Theta(c_i))=m_d(\Theta(c_j))$ if $f(c_i)=f(c_j)$. Conversely, it is proved in \cite[Theorem 3.7]{Ki05} that for each critical portrait $\Theta$ and $r>0$, there exists a unique polynomial $f_r(\Theta)\in \mathcal{S}_d$ with $\Theta(f_r(\Theta))=\Theta$, and the map
\[R_{\mathcal{C}_d}(\Theta):(0,\infty)\to \mathcal{S}_d,\quad r\mapsto f_r(\Theta)\]
is injective. The simple curve $R_{\mathcal{C}_d}(\Theta)\subseteq \mathcal{S}_d$ is called the \emph{parameter ray} at the critical portrait $\Theta$.

\section{Construction of nested disk systems from rational maps}
Following Cui--Tan \cite{CT10}, the key to prove the convergence of a parameter ray is to control the distortion near infinity  of a sequence of univalent maps obtained from the Thurston algorithm. For this purpose, we develop a pullback method to construct $m$-nested, $\lambda$-scattered disk systems near the Julia set for a given rational map and then apply Theorem~\ref{thm_nest}. Our method is similar to that given in \cite[Section 8.3]{CT15} but is a generalization from geometrically finite rational maps to rational maps with no recurrent critical points.

\subsection{Basic settings and notations}\label{subsection_setting}

Throughout the present section, we assume that $f$ is a rational map of degree $d\geq 2$ with a non-empty Fatou set and no recurrent critical points. We also assume that $X_0$ is a finite set in the Julia set $\mc{J}_f$ with the following properties.
\begin{description}
	%	\item[(A1)] $X_0$ is disjoint with the $\omega$-limit sets of recurrent critical points;
	\item[(A1)] $X_0$ contains no recurrent points, i.e., if $x\in X_0$, then $x\not\in \omega(x)$;
	\item[(A2)] $X_0$ contains no periodic points and $\ol{{\rm Orb}(X_0)}$ is disjoint
	from parabolic cycles;
	\item[(A3)] $X_0$ contains all critical points of $f$ in the Julia set;
	\item[(A4)] If $x,y\in X_0$ with $f^n(x)=y$ for some $n\geq1$, then $f(x),\ldots,f^{n-1}(x)\in X_0$.
\end{description}

\begin{rmk}
	To simplify the proof of Theorem~\ref{main_thm}, we have not written $X_0$ as the most general case. In fact, only requirement~(A1) is essential for proving the distortion theorem, namely Theorem~\ref{thm:Cui_Tan}, which is the most crucial result herein. Meanwhile, properties(A2), (A3), and (A4) are merely for simplicity of notations and statements.
	
	Precisely, we claim that Theorem~\ref{thm:Cui_Tan} still holds in the following settings. Let $f$ be any rational map with a non-empty Fatou set and $X_0\subset \olC$ be a finite set satisfying
	\begin{enumerate}
		\item $X_0$ contains no recurrent points;
		\item the set $\ol{{\rm Orb}(X_0)}$ is disjoint from the $\omega$-limit set of recurrent critical points.
	\end{enumerate}

	This condition   is the most general one that we can think of to ensure that Theorem~\ref{thm:Cui_Tan} works in the present method. The proof under this setting must deal with parabolic and repelling cycles, for which we must combine our argument and the one given in \cite[Proposition 8.6]{CT15}.
\end{rmk}

Henceforth, we assume that $X_0$ satisfies properties~(A1)--(A4). To describe points in $X_0$ clearly, we introduce a partial order $\prec$ on $X_0$. For distinct $x,y\in X_0$, we say that $x\prec y$ if $x\in \ol{\tu{Orb}(y)}$. By properties (A1) and (A2), it is easy to check that such $\prec$ is strictly a partial order on $X_0$.

Subsequently, the given set $X_0$ can be decomposed under this partial order. Let $L_1$ be the maximal elements in $X_0$ under the partial order $\prec$. Inductively, for $k\geq 2$, let $L_k$ be the set of maximal elements in $X_0\setminus\cup_{1\leq i\leq k-1} L_i$. Then, there exists $N>0$ such that $$X_0=L_1\sqcup \cdots\sqcup L_N.$$
Meanwhile, given any $x\in X_0$, the set $X_0$ can be divided into two disjoint parts:
$$X_{\prec x}:= \{y\in X_0:y\prec x\}\tu{\quad and\quad }X_{\nprec x}:=X_0\setminus X_{\prec x}.$$ Make sure that $x\notin X_{\prec x}$ but $x\in X_{\nprec x}$ according to (A1).

Let $X_n:=\cup_{0\leq k\leq n}\cup f^{-k}X_0$ and $X_\infty:=\cup_{n\geq 0}X_n$. The \emph{level} of a point $x\in X_\infty$, denoted by $n(x)$, is the minimal integer $n$ such that $f^n(x)\in X_0$. Clearly, such an $n$ is unique.

To simplify the discussion, according to Lemma~\ref{mane_lemma} and by choosing suitable $\delta_0,\eta_0>0$, one can further assume the following.
\begin{description}
	\item[(B1)] The disks $B(x,\delta_0),x\in X_0$ are pairwise disjoint.
	\item[(B2)] For any $x,y\in X_0$ (allowing for $x=y$), the case $B(y,\delta_0)\cap \ol{\tu{Orb}(f(x))}\neq\emptyset$ occurs if and only if $y\in X_{\prec x}$. Equivalently, the condition $y\in X_{\nprec x}$ implies that $B(y,\delta_0)$ is disjoint from $\ol{\tu{Orb}(f(x))}$.
	\item [(B3)] All disks $B(x,\delta_0),x\in X_0,$ are $(\delta,\eta_0)$-backward stable for some $\delta$.	
\end{description}

\subsection{Well-chosen neighborhood sequences of $X_0$}

To construct $m$-nested, $\lambda$-scattered disk systems in a neighborhood of $\mc{J}_f$ labeled by a suitable subset of $X_n$ for every $n\geq0$, we first choose for each $x\in X_0$ a sequence of nested neighborhoods with some geometric and pullback properties.

\begin{lema}\label{lema_x_0}
	One can find constants $C$ and $\delta_1<\delta_0/2$ such that, for each $x\in X_0$, there exists a sequence of nested disks $\mc{N}_x=\{(E_{x,n},O_{x,n},U_{x,n})\}_{n\geq 1}$ contained in $B(x,\delta_1)$, fulfilling the following points:
	\begin{enumerate}
		\item $E_{x,n}\subseteq O_{x,n}\subseteq U_{x,n}$ are nested topological disks with $x\in O_{x,n}$ {but} $x\notin \ol{E}_{x,n}\subseteq\mc{F}_f$;	
		
		\item $\tu{Shape}\,(E_{x,n},\xi_{x,n})\leq C$ for some $\xi_{x,n}\in E_{x,n}$ and $\tu{diam}\, U_{x,n}\leq C\,\tu{diam}\, E_{x,n}$;
		
		\item for any $x,y\in X_0$ with {$y\in X_{\nprec x}$} and any $m,n,k\geq1$, the preimage $f^{-k}(U_{y,m})$ is disjoint from $E_{x,n}$;
		\item for any $x,y\in X_0$ with {$y\in X_{\nprec x}$} and any $k,n \geq1$, the components of $f^{-k}B(y,\delta_1)$ intersecting $\partial U_{x,n}$ are disjoint from $O_{x,n}$.
	\end{enumerate}
\end{lema}

\begin{proof}
	Let $K:=\ol{\cup_{k\geq 0}f^k(X_0)}$. By Lemma~\ref{mane_lemma} and property~(A2), there exist $\delta_z$ and $\eta_z$ for each $z\in K$, such that
	\[{\rm diam}(B_k)<\delta_0/2\quad\text{and}\quad {\rm deg}(f^n|_{B_k})<\eta_z\]
	for any component $B_k$ of $f^{-k}B(z,4\delta_z)$ with $k\geq 0$.
	By the compactness of $K$, one can extract a finite subcover $\{B(z,\delta_{z})\}_{z\in \Sigma}$ of $K$. Let $\kappa_0:=\max\{\eta_z|\ z\in \Sigma\}$.
	
	Using Lemma~\ref{mane_lemma} again, we can find an integer $k_0>0$ such that the diameter of any component of $f^{-k}(B(z,2\delta_z)),z\in\Sigma$ is less than
	$\min_{z\in \Sigma}\delta_z/2$, provided that $k\geq k_0$. It follows that, for each $z\in \Sigma$, the set
	\begin{equation}\label{eq:k0}
	A_z:=B(z,\delta_z)\setminus\ol{\cup_{k\geq k_0}\cup_{y\in\Sigma}f^{-k}B(y,2\delta_y)}
	\end{equation}
	is non-empty. Thus, one can choose a round disk $E_z$ in $A_z\cap \mathcal{F}_f$ (recall that $\mc{F}_f$ is the Fatou set). This means that $E_z$ is disjoint from the preimages $f^{-k}B(z,2\delta_y)$ for all $k\geq k_0,y\in\Sigma$.
	
	Given $x\in X_0$, we begin to construct $\{(E_{x,n},O_{x,n},U_{x,n})\}_{n\geq 1}$ by these $B(z,\delta_z)$ and $E_z$. For every $n\geq1$, the point $x_n:=f^n(x)$ belongs to a $B(z,\delta_z)$ with $z\in \Sigma$. By pulling back, locally at $x$, let $U_{x,n}\supseteq O_{x,n}$ be the components of $f^{-n}B(z, 2\delta_z)$ and $ f^{-n}B(z, \delta_z)$ containing $x$, and $E_{x,n}$ a component of $f^{-n}(E_z)$ contained in $O_{x,n}$. Clearly, property~1 follows directly from this construction.
	
	For property~2, let $n\geq1$ be any integer. Then, there is a point $z\in \Sigma$ such that
	$$f^n: (O_{x,n},U_{x,n})\to (B(z,\delta_z),B(z,2\delta_z))$$
	is a branched covering of degree less than $\kappa_0$. Let $\xi_{x,n}\in E_{x,n}$ such that $f^n(\xi_{x,n})=\xi_z$, the center of disk $E_z$.
	By Lemma~\ref{controlling_holomorphic_distortion}, there exists a constant $C_1$, depending only on $\kappa_0$, such that
	\begin{eqnarray*}
		\tu{Shape}\,(E_{x,n},\xi_{x,n})&\leq& C_1\tu{Shape}\,(E_z,\xi_z)=C_1,\\
		\tu{diam}\,U_{x,n}/\tu{diam}\,E_{x,n}&\leq& C_1\tu{diam}\,B(z,2\delta_z)/\tu{diam}\,E_z.
	\end{eqnarray*}
	Set $C:=\tu{max}\,\{2C_1\tu{diam}\,B(z,2\delta_z)/\tu{diam}\,E_z,z\in \Sigma\}$, then (2)    follows.
	
	Concerning (3),   we only need to prove that it holds for all $m\geq k_0$ and $n,k\geq 1$, where $k_0$ is given in \eqref{eq:k0}, because by replacing each original $U_{y,m},y\in X_0$, with $U_{y,k_0+m}$, property~4 follows. The argument of the claim
	goes by contradiction . Assume that $f^{-k}U_{y,m}\cap E_{x,n}\neq \es$ for some $m\geq k_0$ and $n,k\geq 1$. We first observe that $k>n$: otherwise $U_{y,m}\cap f^{k}U_{x,n}\neq \es$,  and it follows that $$\tu{dist}\,(y,f^k(x))\leq \tu{diam}\,U_{y,m}+\tu{diam}\,f^kU_{x,n} <\delta_0,$$
	which contradicts (B2). Therefore, we have $f^{-(k-n)}U_{y,m}\cap E_z\neq \es$ for some $z\in\Sigma$ with $f^n(E_{x,n})=E_z$. However, by construction, $U_{y,m}$ is a component of $f^{-m}B(z',2\delta_{z'})$ for some $z'\in\Sigma$, and $f^{-(k-n+m)}B(z',2\delta_{z'})$ containing $f^{-(k-n)}U_{y,m}$, is disjoint from $E_z$ because $m\geq k_0$. This leads to a contradiction.
	
	For statement~4  , by applying Lemma~\ref{mane_lemma}, we choose $\delta_1<\delta_0/2$ such that the diameter of any component of $f^{-k}B(y,\delta_1)$, $k\geq 0,y\in X_0$, is less than $\tu{min}\{\delta_z,z\in\Sigma\}$. We claim that (5)  holds under such a choice of $\delta_1$. Indeed, any component $B_k$ of $f^{-k}B(y,\delta_1)$ for some $k\geq 1$ and $y\in X_0$ such that $B_k\cap\partial U_{x,n}\not=\emptyset$ must have $k\geq n$. If not, we have $B(y,\delta_1)\cap \partial f^{k}U_{x,n}\neq \es$, and thus
	$$\tu{dist}(y,f^k(x))<\tu{diam}\,f^{k}U_{x,n}+\tu{diam}\,B(y,\delta_1)<\delta_0.$$
	Therefore, $B(y,\delta_0)\cap\ol{\tu{Orb}(x)}\neq \emptyset$,
	which contradicts (B2). Thus, $f^n(B_k)$ is a component of $f^{-(k-n)}B(y,\delta_1)$. Note that this component intersects the boundary of $f^n(U_{x,n})(=B(z,2\delta_z)\tu{ for some }z\in\Sigma)$. The choice of $\delta_1$ guarantees that ${\rm diam}f^n (B_k)<\delta_z$, thus $f^n(B_k)$ is disjoint from $B(z,\delta_z)$. This implies that $B_k\cap O_{x,n}=\es$ from the definition of $O_{x,n}$. Hence, statement~4 is proved.
	
	By discarding finitely many elements in $\mc{N}_x$ if necessary, we can assume that all disks in $\mc{N}_x$ are contained in $B(x,\delta_1)$ because $\tu{diam}\, U_n\to 0$ as $n\to\infty$. Thus, the lemma is complete.
\end{proof}

\subsection{Construction of nested disk systems from rational maps  }

As we remember, $X_n:=\cup_{0\leq k\leq n}f^{-k}(X_0)$ and $X_\infty=\cup_{n\geq 0}X_n$. We are ready to prove the main proposition, herein.
\begin{prop}\label{prop:nested_disk_system}
	Let $f$ be a rational map and $X_0$ be a finite set with the properties given in Section~\ref{subsection_setting}, and let $W$ be an open set in $\olC$ containing $\ol{X}_\infty$. Then, there is a universal constant $0<\lambda<1$ such that, for any $m>0$ and every integer $n\geq 0$, one can find a subset $\mc{X}_{\leq n}$ of $X_n$ and an
	$m$-nested, $\lambda$-scattered disk system $\{(D_x'',D_x',D_x)\}_{x\in \mc{X}_{\leq n}}$ in $W$ satisfying
	\begin{equation}\label{eq:3}
	\bigcup_{x\in X_n}W_x\subseteq \bigcup_{x\in \mc{X}_{\leq n}} D_x'',
	\end{equation}
	where $W_x,x\in X_n,$ is the component of $f^{-n(x)}W_{x_0}$ containing $x$ with $x_0:=f^{n(x)}(x)$ and $W_{x_0}$ is an arbitrary disk nested in $D_{x_0}''$.
\end{prop}

\begin{proof}
	We adopt the quantities $C$, $\delta_0$, $\delta_1$, and $\mc{N}_x=\{E_{x,n},O_{x,n},U_{x,n}\}_{n\geq 0},x\in X_0$ in Lemma~\ref{lema_x_0}.
	Recall that $X_0$ can been decomposed into $L_1,\cdots, L_N$ according to the partial order $\prec$. Our method of building the desired nested disk systems begins with constructing suitable disk systems $\mc{U}_{k,i},1\leq k\leq N, i\geq0$. The construction proceeds by induction from $L_1$ to $L_N$.
	
	For $k=1$, let $\mu_1=\tu{min}\{\tu{dist~}(\ol{X}_{\infty},\partial W),\delta_1\}$. By property~(B3) and Lemma~\ref{mane_lemma}, one can choose sufficiently small nested disks $(E_x,O_x,U_x)\in \mc{N}_x$ for all $x\in L_1$ such that the disks $U_x$ are $(\mu_1,\eta_0)$-backward stable. To study the pullbacks of such disks $\{U_x\}_{x\in L_1}$, we assume $$\mc{U}_{1,i}=\cup_{x\in L_1}\tu{Comp}(f^{-i}(U_x))\tu{~and~}\mc{X}_{1,i}=\cup_{x\in L_1}f^{-i}(x)\tu{ for all }i\geq 0.$$ Clearly, all elements in families $\mc{U}_{1,i}$ and $\mc{X}_{1,i}$ have the same \emph{level} $i$.
	\vskip 0.3cm
	\tb{Claim~1.} There is a one-to-one correspondence between families $\mc{X}_{1,i}$ and $\mc{U}_{1,i}$.
	\begin{proof}
		Let $x_i\in \mc{X}_{1,i}$ with $x_0=f^{i}(x_i)\in X_0$. Consider the pullback
		$$g_{x_ix_0}:=f^{i}:U_{x_i}\to U_{x_0}\hspace{0.4cm}x_i\mt x_0$$,
		where $U_{x_i}$ is the unique component of $f^{-i}(U_{x_0})$ containing $x_i$. It is important to note that $\tu{Crit}(g_{x_ix_0})\subseteq \{x_i\}$, in other words, the possible critical points of $g_{x_ix_0}$ can be only $x_i$. Otherwise, there is a point other than $x_i$ in $U_{x_i}$ that is iterated to a critical point $c$ of $f$. Under our assumption on $X_0$, we have $x_0\neq c\in X_0$. Because $c\in U_{x_i}\cup\cdots\cup U_{x_1}$, then $\tu{Orb}(c)\cap U_{x_0}\neq \es$. Condition~(B2) implies that $x_0\prec c$, but this is impossible because $x_0$ is a maximal one  in $X_0$. Thus, $U_{x_i}\cap \mc{X}_{1,i}=\{x_i\}$, and the claim holds.	
	\end{proof}
	\vskip 0.3cm
	\tb{Claim~2.} Let $\mc{X}_1:=\cup_{i\geq 0}\mc{X}_{1,i}$ and $\mc{U}_1:=\cup_{i\geq 0}~\mc{U}_{1,i}$. Then for any disks $U_{x_i},U_{y_j}\in \mc{U}_1$ with $x_i\neq y_j\in \mc{X}_1$ and with $i=n(x_i)\leq n(y_j)=j$, we have that $x_i\not\in U_{y_j}$, i.e., the elements of $\mc{U}_1$ with higher levels cannot contain points of $\mc{X}_1$ with lower levels.
	\begin{proof}
		If $i=j$, it is obvious that $U_{x_i}\cap U_{y_j}=\emptyset$ because the disks among $\mc{U}_{1,0}$ of level $0$ are mutually disjoint. If $i<j$, the situation $x_i\in U_{y_j}$ implies that $f^{j}(x_i)\in U_{y_0}=f^jU_{y_j}$ with $y_0\in L_1$. This means that $U_{y_0}$ intersects the orbit of $x_0=f^i(x_i)$. By condition~(B2), we have that $y_0\prec x_0$, which is impossible. Hence, the claim is proved.
	\end{proof}
	Now, for all $x\in L_1$, we pick an open disk $D_x'$ centered at $x$ and nested in $O_x\setminus \ol{E}_x$ such that
	$$\tu{mod~}(O_x\setminus\ol{D'}_x)\geq \eta_0m$$, and we consider the set
	$\Omega_x:=D_x'\setminus\ol{\cup\{U_y\in \mc{U}_1\setminus\{U_x\},U_y\cap \partial D_x'\neq\emptyset\}}$. As a consequence of Claim~2 and of the fact that $\tu{diam~}U_y\to 0$ as $n(y)\to 0$, there exists a unique component of $\Omega_x$ containing $x$. Let $D_x''$ be an open disk centered at $x$ and compactly contained in this component.
	\vskip 0.3cm
	Inductively, for $k=2,\cdots,N$, let $\mu_k$ be the minimum among the following positive values:
	\begin{equation}\label{eq:mu}
	\tu{dist}(y,\partial D_y''),\tu{dist}(\partial D_y'',\partial D_y'),\tu{dist}(\partial U_y,\partial O_y),\tu{dist}(\mc{J}_f,E_y),y\in L_1\cup\cdots\cup L_{k-1}.
	\end{equation}
	Choose sufficiently small $(E_x,O_x,U_x)\in \mc{N}_x,x\in L_k$ such that the disks $U_x$ are $(\mu_k,\eta_0)$-backward stable.
	Moreover, we aim to collect some ``nice’’ components (i.e., they satisfy Claims~1 and 2) among the family $\cup_{x\in L_k}\cup_{i\geq 0}\tu{Comp}(f^{-i}U_x)$.
	
	The construction of families $\mc{U}_{k,i}$ and $\mc{X}_{k,i}$ is done by induction on level $i$. 
	To do this, we begin by setting $\mc{U}_{k,0}=\{U_x\}_{x\in L_k}$ and $\mc{X}_{k,0}=L_k$. Then inductively, for $i=1,2,\cdots$, if a component $U$ of $f^{-1}U_{x_{i-1}}$ with $U_{x_{i-1}}\in \mc{U}_{k,i-1}$ and $x_{i-1}\in\mc{X}_{k,i-1}$ covers a point $y\in L_1\cup \cdots \cup L_{k-1}$, then
	\begin{equation}\label{eq:hidden}
	L_k\ni f^{i-1}(x_{i-1})=x_0\prec y\tu{~~~and~~~}\tu{diam}(U)\leq \mu_k\leq \tu{dist}(y,\partial D_y'').
	\end{equation}
	It follows that $U\subseteq D_y''.$ We refer to such elements in $f^{-1}\mc{U}_{k,i-1}$ as \emph{hidden} components, with the others in $f^{-1}\mc{U}_{k,i-1}$ being \emph{non-hidden} components. We point out that all elements in $\mc{U}_{1}$ and $\mc{U}_{k,0}$ are non-hidden components.
	Thus, we can define $\mc{U}_{k,i}$ to be the collection of all \tb{non-hidden}  elements within the family $f^{-1}\mc{U}_{k,i-1}$, and that all of them have the same level $i$.
	\vskip 0.3cm
	\tb{Claim~1$'$.} Let $\mc{X}_{k, i}:=f^{-i}(L_k)\cap\left(\cup\{U\in\mc{U}_{k,i}\}\right)$. Then, there is a one-to-one correspondence between the families $\mc{X}_{k,i}$ and $\mc{U}_{k,i}$.
	\begin{proof}
		We must prove that each component $U$ in $\mc{U}_{k,i}$ contains exactly one point in $f^{-i}L_k$. Let $U_{x_0}:=f^{i}U\in \mc{U}_{k,0}$ with $x_0\in L_k$. Consider the pullback
		$$g_{x_ix_0}:=f^i:U_{x_i}\to U_{x_0}\hspace{0.4cm} x_i\mt x_0,$$
		where $x_i$ is an arbitrary point among the preimage of $g^{-1}_{x_ix_0}(x_0)$ and $U_{x_i}:=U$. It is enough to show that $\tu{Crit}(g_{x_ix_0})\subseteq \{x_i\}$ or equivalently that the possible critical points of $g_{x_ix_0}$ can be only $x_i$. To see this, if not , the orbit
		$$U_{x_i}\to U_{x_{i-1}}\to\cdots\to U_{x_1}(\to U_{x_0})$$
		will cover a critical point $c$ of $f$, which avoids the set $\{x_i,\cdots,x_1\}$, say $c\in U_{x_{i_0}}$. Since $c$ is contained in $X_0$.  Thus, by (B2) we have that $x_0\prec c$. This means that $U_{x_{i_0}}$ contains a point in $L_1\cup \cdots\cup L_{k-1}$. Thus, $U_{x_{i_0}}$ is a hidden component, which is contradicting because $U_{x_{i_0}}\in \mc{U}_{k,i_0}$. Hence, the claim is complete.
	\end{proof}
	\vskip 0.3cm
	\tb{Claim~2$'$.} Let $\mc{X}_k:=\cup_{i\geq 0}\mc{X}_{k,i}$ and $\mc{U}_k:=\cup_{i\geq 0}~\mc{U}_{k,i}$. Then, for any disks $U_{x_i},U_{y_j}\in \mc{U}_1\cup\cdots\cup \mc{U}_{k}$ with $x_i\neq y_j\in \mc{X}_1\cup\cdots\cup \mc{X}_k$ and with $i=n(x_i)\leq n(y_j)=j$, we have that $x_i\not\in U_{y_j}$, i.e., the elements of $\mc{U}_k$ with higher levels cannot contain points of $\mc{X}_k$ with lower levels.
	\begin{proof}
		If $i=j$, it follows directly  from the fact that $U_{x_0}(=f^{i}U_{x_i})$ and $U_{y_0}(=U_{y_j})$ are disjoint. If $i<j$, assume $x_i\in U_{y_j}$, whereupon $f^{j-i}(x_0)\in U_{y_0}$. It follows that $y_0\prec x_0$ and $f^{i}U_{y_j}$, which is equal to $U_{y_{j-i}}(\in\mc{U}_{j-i})$, covers $x_0$ in $X_0$. Thus, $U_{y_{j-i}}$ is a hidden component, which is a contradiction. Hence, the claim is completed.
	\end{proof}
	
	Now for all $x\in L_k$ we pick an open disk $D_x'$ centered at $x$ and nested in $O_x\setminus \ol{E}_x$ such that
	$\tu{mod~}(O_x\setminus\ol{D'}_x)\geq \eta_0m.$ Then, consider the set
	\begin{equation}\label{eq:omega}
	\Omega_x:=D_x'\setminus\ol{\cup\{U_y\in \mc{U}_1\cup\cdots\cup \mc{U}_k\setminus\{U_x\},U_y\cap \partial D_x'\neq\emptyset\}}.
	\end{equation}
	As a consequence of Claim~2$'$ and of the fact that $\tu{diam~}U_y\to 0$ as $n(y)\to 0$, there exists a unique component of $\Omega_x$ containing $x$. Let $D_x''$ be an open disk centered at $x$ and compactly contained in this component. 	
	
	After the above construction from $k=1$ to $N$, we obtain the sets $\mc{X}:=\mc{X}_1\cup \cdots\cup\mc{X}_N$, $\mc{U}:=\mc{U}_1\cup\cdots\cup\mc{U}_N$ and $\{(D_x'',D_x',E_x,O_x,U_x)\}_{x\in X_0}$. We claim and check the following statements.
	\vskip 0.3cm
	\tb{Claim~3.} The families $\mc{X}$ and $\mc{U}$ are forward invariant, that is, $\mc{X}\subseteq f(\mc{X})$ and $\mc{U}\subseteq f(\mc{U})$.
	\begin{proof}
		This is a direct consequence of the fact that the images of any non-hidden components in $\mc{U}$ with levels of at least one are still non-hidden.
	\end{proof}
	
	For any $x_i\in\mc{X}$ of level $i\geq 1$, we define $(D_{x_i}'',D_{x_i}',O_{x_i})$ as the pullback of $(D_{x_0}'',D_{x_0}',O_{x_0})$ under the covering $g_{x_ix_0}:U_{x_i}\to U_{x_0}$, while $E_{x_i}$ is an arbitrary component of $g_{x_ix_0}^{-1}(E_{x_0})$.
	\vskip 0.3cm
	\tb{Claim~4.} For any $n\geq 0$, let $\mc{X}_{\leq n}:=\{x\in\mc{X}:n(x)\leq n\}$. Then, we have
	$$\cup_{x\in X_0}\cup_{0\leq i\leq n}f^{-i}D_x''\subseteq \cup_{y\in\mc{X}_{\leq n}}D_y''.$$
	\begin{proof}
		For any $x_0\in X_0,0\leq i\leq n$, let $D''_{x_i}$ and $U_{x_i}$ be components of $f^{-i}D_{x_0}''$ and $f^{-i}U_{x_0}$, respectively, such that $f^{i}(x_i)=x_0$ and $x_i\in D''_{x_i}\subseteq U_{x_i}$. If $U_{x_i}$ is hidden, by definition it contains a point $y\in X_0=\mc{X}_{\leq 0}$ such that $x_0\prec y$. Thus, $D''_{x_i}\subseteq U_{x_i}\subseteq D_{y}''$ from \eqref{eq:hidden}. By induction, for $1\leq k\leq n$, we can assume that $f^k(U_{x_i})(=U_{x_{i-k}})$ is hidden and $f(U_{x_i})(=U_{x_{i-1}})\subseteq D''_y\subseteq U_y$ for some $y\in \mc{X}_{k-1}$. Let $y'\in D_{y'}''\subseteq U_{y'}$ be the pullback of $y,D_y'',U_y$ under the mapping $f:U_{x_i}\to U_{x_{i-1}}$ such that $U_{x_i}\subseteq D_{y'}''$. If $U_{y'}$ is non-hidden, we are done because  $y'\in\mc{X}_{\leq k}$. Otherwise, $U_{y'}$ is hidden and contained in some $D_z''$ with $z\in \mc{X}_{\leq 0}$ as per the above discussion. It follows that $D_{x_i}\subseteq U_{x_i}\subseteq D''_{y'}\subseteq U_{y'}\subseteq D_z''$. Hence, the claim is complete.
	\end{proof}
	\vskip 0.3cm
	\tb{Claim~5.} For any distinct points $x_i,y_j\in \mc{X}$ of levels $i,j$ with $i\leq j$, consider the disks $$(D_{z}'',D_{z}',E_{z},O_{z},U_{z}),{z\in\{x_i,y_j\}}.$$ We have the following.
	\begin{itemize}
		\item[(1)] If $U_{y_j}\cap D_{x_i}''\neq \es$, then $U_{y_j}\subseteq D_{x_i}'$.
		\item[(2)] If $U_{y_j}\cap \partial U_{x_i}\neq \es$, then $U_{y_j}\cap O_{x_i}=\es$.
		\item[(3)] $U_{y_j}\cap E_{x_i}=\es$.
		\item[(4)] $\tu{Shape}(E_{x_i},\xi_{x_i})<C$ for some $\xi_{x_i}\in E_{x_i}$ and $\tu{diam}(U_{x_i})\leq C\tu{diam}(E_{x_i})$, where $C$ is a universal constant that is independent of $x_i$.
		\item[(5)] $\tu{mod}(O_{x_i}\setminus\ol{D'}_{x_i})\geq m$.
	\end{itemize}
	\begin{proof}
		Assume $x_0:=f^i(x_i)\in L_{k_i}$ and $y_0=f^j(y_j)\in L_{k_j}$.
		
		(1) By definition, after the action of $f^{i}$, we have $U_{y_{j-i}}\cap D_{x_0}''\neq \es$. If $k_i<k_j$.  According to the choice of $\mu_k$, it follows that $\tu{diam}(U_{y_{j-i}})\leq \tu{dist}(\partial D_{x_0}'',\partial D_{x_0}')$. This implies that $U_{y_{j-i}}\subseteq D_{x_0}'$. Otherwise, we have $k_i\geq k_j$. By the definition of $\Omega_{x_0}$ [see \eqref{eq:omega}], the component $U_{y_{j-i}}$ cannot meet both of the boundaries of $ D_{x_0}'$ and $D_{x_0}''$. It follows that $U_{y_{j-i}}\subseteq D'_{x_0}$. Thus, $(1)$ holds.
		
		(2) We have $U_{y_{j-i}}\cap \partial U_{x_0}\neq \emptyset$. It is enough to show that $U_{y_{j-i}}\cap O_{x_0}=\es$. If $y_0\prec x_0$, then $k_i<k_j$. By the choice of $\mu_{k_j}$ [see \eqref{eq:mu}], we have that $\tu{diam}(U_{y_{j-i}})<\tu{dist}(\partial U_{x_0},\partial O_{x_0})$. Thus, $U_{y_{j-i}}\cap O_{x_0}=\es$. Otherwise, we have $y_0\in X_{\nprec x_0}$. Because $U_{y_0}\subseteq B(y_0,\delta_1)$, from Lemma~\ref{lema_x_0}(4), one can obtain $U_{y_{j-i}}\cap O_{x_0}=\es$. Hence, statement (2) is proved.
		
		(3) It is enough to show that $U_{y_{j-i}}\cap E_{x_0}=\es$. If $y_0\prec x_0$, then $k_i<k_j$. The choice of $\mu_{k_j}$ gives that $\tu{diam}(U_{y_{j-i}})\leq \mu_{k_j}\leq\tu{dist}(E_{x_0},\mc{J}_f)$. Because $y_{j-i}\in\mc{J}_f$ and $E_{x_0}\subseteq\mc{F}_f$, we have $U_{y_{j-i}}\cap E_{x_0}=\es$. Otherwise, we have $y_0\in X_{\nprec x_0}$. By Lemma~\ref{lema_x_0}(3), the conclusion follows.
		
		Concerning statements (4) and (5), we need consider only the branched covering
		$$g_{x_ix_0}:=f^i:(\xi_{x_i},E_{x_i}, D'_{x_i},O_{x_i},B_{x_i}'',B'_{x_i})\to (\xi_{x_0},E_{x_0},D'_{x_0},O_{x_0},B(x_0,\delta_1),B(x_0,\delta_0))$$, 
		where $B_{x_i}'$ and $B_{x_i}''$ are the components of $f^{-i}B(x_0,\delta_1)$ and $f^{-i}B(x_0,\delta_0)$ containing $E_{x_i}$, respectively . Note that
		$$\tu{deg}(g_{x_ix_0})\leq \eta_0\tu{ and }\tu{mod}(B(x_0,\delta_0)\setminus\ol{B(x_0,\delta_1)})\geq \frac{\tu{log~} 2}{2\pi}$$ according to the choice of $\delta_0$ and $\delta_1$. We apply Lemmas~\ref{controlling_holomorphic_distortion} and \ref{lema_x_0}(2) to $g_{x_ix_0}$. Then, statements (4) and (5) hold.
	\end{proof}
	
	Moreover, we aim to construct a nested disk system $\{D_x\}_{x\in \mc{X}_{\leq n}}$ for any given $n\geq 0$. Recalling $\mc{X}_{\leq n}:=\{x\in\mc{X},n(x)\leq n\}$. The construction proceeds by induction.
	
	For every $x\in \mc{X}_{\leq n}$ with $n(x)=n$, let $D_x:=U_x$. Inductively, for each $x\in \mc{X}_{\leq n}$ with its level $k:=n(x)$ running from $n-1$ to 0, we set
	$$b(x):=\{y\in \mc{X}_{\leq n}: k+1\leq n(y)\leq n\tu{ and }D_y\cap \partial U_x\neq \es\}.$$
	According to Claim~5(2), any disks $D_y,y\in b(x),$ are disjoint from $O_x$. Then, there exists a unique component of $U_x\setminus\ol{\cup_{y\in b(x)}U_y}$ such that it contains $O_x$. Such a component is an open topological disk and is denoted by $D_x$. It follows immediately that $x\in D''_x\subseteq D_x'\subseteq O_x\subseteq D_x$.
	\vskip 0.3cm
	\tb{Claim~6.} The disk system $\{(D_x'',D_x',D_x)\}_{x\in \mc{X}_{\leq n}}$ is $m$-nested, $\lambda$-scattered in $W$.
	\begin{proof}
		First, from the above construction, it is clear that the disks $\{D_x\}_{x\in\mc{X}_{\leq n}}$ are nested. Moreover, by Claim~2$'$, for distinct $x_i,y_j\in \mc{X}_{\leq n}$ with levels $i,j$ satisfying $0\leq i\leq j\leq n$, we have either $D_{x_i}\cap D_{y_j}=\emptyset$ or $D_{y_j}\subseteq D_{x_i}$.
		
		Next, by Claim~5(5), the $m$-nested property follows directly from the inequality
		$$\tu{mod}(D_x\setminus\ol{D'}_x)\geq \tu{mod}(O_x\setminus\ol{D'}_x)\geq m, \tu{ for all }x\in \mc{X}_{\leq n}.$$
		
		We are yet to check the $\lambda$-scattered property. To do this, let $V_x$ be the union of all disks $D_y$ such that $D_y\subseteq D_x$. According to Claims~2$'$ and 5, we must have $n(y)>n(x)$ and $D_y\cap E_x=\emptyset$ for such $y$ . Combining with Claim~5(4) and Lemma~\ref{lema_area}, for $x\in \mc{X}_{\leq n}$, we compute
		\begin{equation}\notag
		\begin{split}
		\tu{Area}(\rho_*,h(V_x))&\leq\tu{Area}(\rho_*,h(D_x))-\tu{Area}(\rho_*,h(E_x))\\
		&\leq\tu{Area}(\rho_*,h(D_x))-\lambda'\tu{Area}(\rho_*,h(U_x))\\
		&\leq\tu{Area}(\rho_*,h(D_x))-\lambda'\tu{Area}(\rho_*,h(D_x))\\
		&=\lambda\tu{Area}(\rho_*,h(D_x)),
		\end{split}
		\end{equation}
		where $\lambda'$ is the constant, as stated in Lemma~\ref{lema_area}. Hence, the claim is proved.
	\end{proof}
	Combining with Claim~4 , we thus complete the entire proof.
\end{proof}

The following distortion theorem is the most crucial result in the present paper. It is a direct consequence of Theorem~\ref{thm_nest} and Proposition~\ref{prop:nested_disk_system}.
\begin{thm}\label{thm:Cui_Tan}
	Let $f$ be a rational map with a non-empty Fatou set and no recurrent critical points, and let $X_0$ be a finite set satisfying properties~(A1)--(A4). Let $V$ be a Jordan disk in $\mathcal{F}_f$ containing three distinct points $z_1,z_2,z_3$. Then, for any $\epsilon>0$, there exists $\delta>0$ such that
	$$\mathop{\tu{sup}}_{z\in V}^{}\,\{\tu{dist}\,(\phi(z),z)\}\leq \epsilon$$
	for any $n\geq 0$ and any univalent map $\phi:\olC\setminus\ol{\cup_{x\in X_n} B_x}\to \olC$ fixing $z_1,z_2,z_3$ and any $n\geq0$, where $B_x$ is the component of $f^{-n(x)}B(x',\delta)$ containing $x$ with $x':=f^{n(x)}(x)\in X_0$.
\end{thm}
\section{Proof of the main result}

\begin{proof}[Proof of Theorem~\ref{main_thm}]
	Figure~\ref{fig_relation} illustrates the idea of the proof. We will present the proof in four steps.
	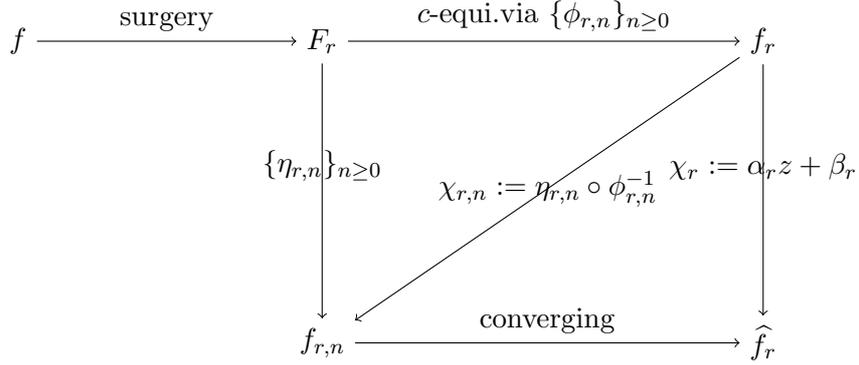
\begin{figure}[htpb]
		\centering
		\begin{tikzpicture}
		
		\node (fr) at (-5,2) {$f$};
		\node (Frdelta) at (-1,2) {$F_{r}$};                       \node (frdelta) at (4.8,2)  {$f_{r}$};
		\node (fdeltan) at (-1,-2){$f_{r,n}$};                       \node (hatfdelta) at (4.8,-2) {$\wh{f}_r$};
		\draw [->] (fr) -- node [below]{} node [above]{surgery} (Frdelta);
		\draw [->] (Frdelta) -- node [below]{} node [above]{$c$-equi.via $\{\phi_{r,n}\}_{n\geq 0}$}(frdelta);
		\draw [->] (Frdelta) -- node [above]{$\{\eta_{r,n}\}_{n\geq0}$}node [below]{}(fdeltan);
		\draw [->] (frdelta) -- node [above]{$\chi_{r}:=\alpha_r z+\beta_r$} node [below]{}(hatfdelta);
		\draw [->] (fdeltan) -- node [above]{converging} node [below]{}(hatfdelta);
		\draw [->] (frdelta) -- node {$\chi_{r,n}:=\eta_{r,n}\circ\phi_{r,n}^{-1}$} (fdeltan);
		%\ruler{7}{3}
		\end{tikzpicture}
		\caption{Relation of maps  in the proof of Theorem~$\ref{main_thm}$.}
	\end{figure}\label{fig_relation}

	\noindent\textbf{Step~I.} Surgery to construct topological polynomials $F_{r}$.\\
	
	Denoted by $V_r$, is the bounded component surrounded by the Jordan curve $\psi_f^{-1}(e^{r}e^{i[0,2\pi]})$.
	For each critical value $v$ of $f$, there is a unique argument $\theta$ in $m_d(\Theta)$ such that the external ray $R_f(\theta)$ lands at $v$. Subsequently, we choose a point $v_r:=\psi_f^{-1}(e^re^{2\pi i\theta})\in R_f(\theta)$ and a Jordan disk $W_{r,v}$ containing $v$ such that
	\begin{itemize}
		\item $\tu{diam}\,W_{r,v}\to 0$ as $r\to 0$,
		\item $W_{r,v}\subseteq f(V_r)$,
		\item the disks $W_{r,v}$ for distinct critical values are pairwise disjoint.
	\end{itemize}
	
	Let $\zeta_r:\olC\to\olC$ be a quasi-conformal map that sends each critical value $v$ to $v_r$ within $W_{r,v}$ and is the identity outside the finitely many disks $W_{r,v}$. The topological polynomial $F_r$ is defined as
	$F_r:=\zeta_r\circ f$. For each critical point $c$ with $f(c)=v$, we denote it by $W_{r,c}$ (the component of $f^{-1}W_{r,v}$ containing $c$). Let $W_r:=\cup_{c\in\tu{Crit}(f)}W_{r,c}$. It is easy to check the following facts:
	\begin{itemize}
		\item $F_r$ is quasi-regular, and $F_r=f$ on $\olC\setminus W_r$;
		\item $\tu{Crit}(F_r)=\tu{Crit}(f)\subseteq \partial F_r^{-1}V_r$;
		\item all critical points of $F_r$ escape to infinity at the same speed.
	\end{itemize}

	\vskip 0.2cm
	
	\noindent\textbf{Step~II.} The topological polynomial $F_{r}$ is $c$-equivalent to $f_r$.\\ %via $(\phi_{r,n},\phi_{r,n+1})$, $n\geq 0$.\\
	
	As we mentioned in Section~\ref{sec:critical-portrait}, for any $r>0$, there exists a unique polynomial $f_r(\Theta)\in \mathcal{S}_d$ with the critical portrait $\Theta$ and escaping rate $r$ for all critical values. For simplicity, we write $f_r(\Theta)$ as $f_r$ because $\Theta$ is given. We will show that the topological polynomial $F_r$ constructed in Step~I has dynamics that are very similar to those of $f_r$.
	
	For any $r>0$, we denote by $V'_r$ the bounded component surrounded by the Jordan curve $\psi_{f_r}^{-1}(e^re^{i[0,2\pi]})$.
	Set $U_r:=F_r^{-1}V_r$ and $U'_r:=f_r^{-1}V_r'$. Note that both of them  comprise $d$ Jordan disks because all critical values of $F_r$ (resp.\,$f_r$) are contained in $\partial V_r$ (resp.\,$\partial V_r'$). We define a quasi-conformal homeomorphism $\phi_{r,0}:\olC\to \olC$ by
	\begin{equation}\label{eq:phi}
	\phi_{r,0}:=
	\begin{cases}
	\psi_{f_{r}}^{-1}\circ\psi_{f}(z) & \text{if $z \in \olC\setminus V_r,$} \\
	h_r(z) & \text{if $z\in {V}_r,$}\\
	\end{cases}
	\end{equation}
	where $h_r:V_r\to V_r'$ is a quasi-conformal map that agrees with $\psi_{f_r}^{-1}\circ\psi_f(z)$ on $\partial V_r$. Because $\psi_f$ and $\psi_{f_r}$ are B\"{o}ttcher coordinates of $f$ and $f_r$, respectively, the following commutative graph holds:
	\begin{equation}\label{eq:commutative3}
	\begin{CD}
	{\olC\setminus V_r}@> {\phi_{r,0}} >> {\olC\setminus V_r'}\\
	@ V{F_{r}} VV @ VV {f_{r}}V\\[-3pt]
	{\olC\setminus V_r}@> {\phi_{r,0}} >> {\olC\setminus V_r'.}\\
	\end{CD}
	\end{equation}
	
	We claim that $\phi_{r,0}$ can be lifted along $F_r$ and $f_r$ by the following graph:
	\[
	\begin{CD}
	{( U_r,\ \olC\setminus U_r)}@> {\phi_{r,1}} >> {( U_r',\ \olC\setminus U_r')}\\
	@ V{F_{r}} VV @ VV {f_{r}}V\\
	{(V_r,\ \olC\setminus V_r)}@> {\phi_{r,0}} >> {(V_r',\ \olC\setminus V_r').}\\
	\end{CD}
	\]
	Indeed, because $F_r$ (resp.\,$f_r$) sends any components of $U_r$ (resp.\,$U_r'$) homeomorphically onto $V_r$ (resp.\,$V_r'$), there exists a unique lift $\wt{h}_r:U_r\to U_r'$ of $h_r$ (defined in \eqref{eq:phi}). In contrast, the restriction of $F_r$ (resp.\,$f_r$) on $\olC\setminus {U_r}$ (resp.\,$\olC\setminus {U_r'}$) is a $d:1$ branched covering with only one critical point at infinity. Together with \eqref{eq:commutative3}, we can choose a unique lift $$\phi_{r,1}:\olC\setminus{U_r}\to\olC\setminus{U_r'}$$
	of $\phi_{r,0}:\olC\sm {V_r}\to \olC\sm {V_r'}$ such that $\phi_{r,1}=\phi_{r,0}$ on $\olC\setminus V_r$. It is easy to see that $\phi_{r,1}=\wt{h}_r$ on ${\partial U_r}$. Thus, $\phi_{r,1}$ can be extended to a quasi-conformal homeomorphism on $\olC$ by setting $\phi_{r,1}:=\wt{h}_r$ on $U_r$, and it follows that $\phi_{r,0}\circ F_r=f_r\circ \phi_{r,1}$ on $\olC$.
	
	By this construction, the homeomorphism $\phi_{r,0}$ is holomorphic and coincides with $\phi_{r,1}$ on $\olC\setminus V_r$, which contains the post-critical set of $F_r$. By contrast, within $V_r$, $\phi_{r,0}$ and $\phi_{r,1}$ are isotopic rel.  ${\rm post}(F_r)$ by Alexander's trick. Following \cite{CT11}, such a pair of branched coverings $F_r, f_r$ is called a \emph{$c$-equivalence} via $\phi_{r,0},\phi_{r,1}$.
	
	Let $H_0:\olC\times [0,1]\to \olC$ rel.  $\olC\setminus V_r$ be an isotopy from $\phi_{r,0}$ to $\phi_{r,1}$. Inductively,
	for $n\geq 0$, there exists a unique lift $H_{n+1}$ of $H_{n}$ such that $H_{n+1}(\cdot,0)=\phi_{r,n}$. Let $\phi_{r,n+1}:=H_{n+1}(\cdot,1)$. Then we get a sequence of homeomorphisms $\{\phi_{r,n}\}_{n\geq 0}$ such that
	\begin{itemize}
		\item $\phi_{r,n}\circ F_r=f_r\circ \phi_{r,n+1}$ [see \eqref{eq:commutative1}],
		\item $\phi_{r,n}=\phi_{r,n+1}$ on $\bigcup_{0\leq k\leq n}F_r^{-k}(\olC\setminus V)$.
	\end{itemize}
	\begin{equation}\label{eq:commutative1}
	\begin{CD}
	{\olC} @> {\cdots} >> {\olC}@>{F_r}>>{\olC}@>{\cdots}>>{\olC}@>{F_r}>>{\olC}@>{F_r}>>{\olC} \\
	@ VV{\cdots}V @ VV {\phi_{r,n+1}} V@VV {\phi_{r,n}}V@VV {\phi_{r,2}}V@VV {\phi_{r,1}}V@VV {\phi_{r,0}}V \\
	{\olC} @> {\cdots} >> {\olC}@>{f_r}>>{\olC}@>{\cdots}>>{\olC}@>{f_r}>>{\olC}@>{f_r}>>{\olC} \\
	\end{CD}
	\end{equation} 
	
	\vskip 0.2cm
	\noindent \textbf{Step~III.} Thurston algorithm on $F_r$\\%to obtain $\{\eta_{r,n}\}_{n\geq 0}$ and polynomials $\{f_{r,n}\}_{n\geq 0}$.\\
	
	Let $\eta_{r,0}:=\tu{id }$. Then, $F_r\circ\eta_{r,0}$ defines a complex structure on $\olC$ by pulling back the standard complex structure. The uniformization theorem guarantees the existence of a unique homeomorphism $\eta_{r,1}:\olC\to \olC$, normalized by fixing $a_1,a_2,\infty$ , with $a_1,a_2$ close to infinity, such that $f_{r,0}:=\eta_{r,0}\circ F_{r}\circ\eta_{r,1}^{-1}$ is holomorphic. Note that $\eta_{r,1}$ is holomorphic except on $W_r$.
	
	Recursively, there exist a quasi-conformal map $\eta_{r,n+1}$ and a polynomial $f_{r,n+1}$ for $n\geq 0$ such that
	
	\begin{itemize}
		\item[(1)] $\eta_{r,n}\circ F_{r}=f_{r,n}\circ\eta_{r,n+1}$ [see \eqref{eq:commutative2}],
		\item[(2)] $\eta_{r,n+1}$ is univalent on $\olC\sm \bigcup_{0\leq i\leq n} F_r^{-i}(W_r)=\olC\sm \bigcup_{0\leq i\leq n} f^{-i}(W_r)$,
		\item[(3)] $\eta_{r,n+1}$ fixes $a_1$, $a_2$, and infinity.
	\end{itemize}
	\begin{equation}\label{eq:commutative2}
	\begin{CD}
	{\olC} @> {\cdots} >> {\olC}@>{F_r}>>{\olC}@>{\cdots}>>{\olC}@>{F_r}>>{\olC}@>{F_r}>>{\olC} \\
	@ VV{\cdots}V @ VV {\eta_{r,n+1}} V@VV {\eta_{r,n}}V@VV {\eta_{r,2}}V@VV {\eta_{r,1}}V@VV {\eta_{r,0}}V \\
	{\olC} @> {\cdots} >> {\olC}@>{f_{r,n}}>>{\olC}@>{\cdots}>>{\olC}@>{f_{r,1}}>>{\olC}@>{f_{r,0}}>>{\olC} \\
	\end{CD}
	\end{equation} 
	
	Let $V\subseteq \olC\setminus\ol{V_r}$ be an open disk containing $a_1,a_2,\infty$. Because of properties (2) and (3) mentioned above, Theorem~\ref{thm:Cui_Tan} gives a crucial distortion estimate, namely
	\begin{equation}\label{equ:distortion}
	\mathop{\tu{sup}}^{}_{n\geq 0,z\in V}\tu{dist}\,(\eta_{r,n}(z),z)\to0\tu{\; as \;}r\to0.
	\end{equation}
	Note that $F_r=f$ on $\olC\setminus V_r\supset V$, then \eqref{eq:commutative2} and \eqref{equ:distortion} imply
	\begin{equation}\label{equ:distortion_f}
	\mathop{\tu{sup}}^{}_{n\geq 0,z\in V}\tu{dist}\,(f_{r,n}(z),f(z))\to0\tu{\; as \;}r\to0.
	\end{equation}
	
	\vskip 0.2cm
	\noindent\textbf{Step~IV.} The parameter ray $R_{\mathcal{C}_d}(\Theta)$ lands at $f$. \\%$\{f_{r,n}\}_{n\geq 0}$, $\{\wh{f}_r\}_{r>0}$ and $\{f_r\}_{r>0}$.
	
	Let $z_{r,i}:=\phi_{r,n}(a_i),i\in\{1,2\}$. They are independent of $n$  because $\phi_{r,n}=\phi_{r,0}$ on $\mathbb{C}\setminus V_r\supset V$ for all $n\geq 0$.
	It follows that each map $\chi_{r,n}:=\eta_{r,n}\circ\phi^{-1}_{r,n}$ with $n\geq0$ sends $\infty,z_{r,1},z_{r,2}$ to $\infty, a_1,a_2$, respectively. Combining \eqref{eq:commutative1} and \eqref{eq:commutative2}, we obtain the following commutative graph:
	\[\begin{CD}
	{\olC} @> {\cdots} >> {\olC}@>{f_r}>>{\olC}@>{\cdots}>>{\olC}@>{f_r}>>{\olC}@>{f_r}>>{\olC} \\
	@ VV{\cdots}V @ VV {\chi_{r,n+1}} V@VV {\chi_{r,n}}V@VV {\chi_{r,2}}V@VV {\chi_{r,1}}V@VV {\chi_{r,0}}V \\
	{\olC} @> {\cdots} >> {\olC}@>{f_{r,n}}>>{\olC}@>{\cdots}>>{\olC}@>{f_{r,1}}>>{\olC}@>{f_{r,0}}>>{\olC} \\
	\end{CD}\]
	Note that $\chi_{r,0}$ is quasi-conformal on $\olC$ and holomorphic on $\olC\setminus V_r$, then the diagram implies that each $\chi_{r,n}$ is quasi-conformal on $\olC$ with a uniformly bounded dilation $K_r$ and is holomorphic on $\bigcup_{0\leq k\leq n}f_r^{-k}(\olC\setminus V_r)$. Thus $\{\chi_{r,n}\}_{n\geq0}$ is a normal family because each $\chi_{r,n}$ sends $\infty,z_{r,1},z_{r,2}$ to $\infty, a_1,a_2$ respectively.
	
	Let $\chi_r:\olC\to\olC$ be any limit of a subsequence of $\{\chi_{r,n}\}$. Then, it is $K_r$-quasi-conformal on $\olC$ and is holomorphic on $\bigcup_{k\geq 0}f_r^{-k}(\mathbb{C}\setminus V_r)$, which is exactly the entire Fatou set of $f_r$. It follows that $\chi_r$ is holomorphic on $\olC$ because $\mathcal{J}_{f_r}$ is removable. Thus, one can write $\chi_r$ as the affine map
	\begin{equation}\notag
	\chi_r:z\mt\alpha_{r}z+\beta_r,
	\end{equation}
	which maps $z_{r,i}$ to $a_i$, $i=1,2$. Applying the argument above to any convergence   subsequence of $\{\chi_{r,n}\}_{n\geq0}$, we have that the entire sequence $\{\chi_{r,n}\}_{n\geq 0}$ uniformly converges to $\chi_{r}$ on $\olC$.
	As a consequence, the Thurston sequence $\{f_{r,n}\}_{n\geq 0}$ uniformly converges to the polynomial
	\begin{equation}\label{eq:conjugacy}
	\wh{f}_r:=\chi_r\circ f_r\circ\chi_r^{-1}.
	\end{equation}
	Estimate~\eqref{equ:distortion_f} implies that $\wh{f}_r\to f$ as $r\to0$ on $V$, and Lemma~\ref{convergence_of_rational_maps} further tells us that $\wh{f}_r\to f$ as $r\to0$ on $\olC$.
	%Assume that
	%$$\wh{f}_r(z)=a_{r,d}z^d+a_{r,d-1}z^{d-1}+\cdots+a_{r,0}.$$
	
	The landing of parameter ray $R_{\mathcal{C}_d}(\Theta)$ at $f$ is equivalent to $f_r\to f$ as $r\to0$, for which we only need to prove that
	\[\text{$\sup_{z\in\olC}\ \tu{dist}\,(\wh{f}_r(z),f_r(z))\to 0$\quad as\quad $r\to 0$}\]
	by the discussion above.
	By \eqref{eq:conjugacy}, it is enough to verify that $\chi_r\to {\rm id}$ as $r\to 0$, i.e., $\alpha_r\to 1$ and $\beta_r\to 0$ as $r\to0$. We write
	$$\wh{f}_r(z)=c_{r,d}z^d+c_{r,d-1}z^{d-1}+\cdots+c_{r,0}.$$
	Note that $f_r$ is monic and centered, then a simple computation implies
	\[
	\begin{cases}
	c_{r,d}\cdot\alpha_r^{d-1}=1, & \\
	d\cdot c_{r,d}\cdot\beta_r\cdot\alpha_{r}^{d-1}+c_{r,d-1}\cdot\alpha_r^{d-1}=0. & \\
	\end{cases}
	\]
	Because $\wh{f}_r\to f$, which is monic and centered, as $r\to 0$, it follows that
	$c_{r,d}\to 1$ and $c_{r,d-1}\to 0$ as $r\to 0$. Thus, $\alpha_r$ tends to the $(d-1)$-th root of unity   and $\beta_r\to 0$ as $r\to0$.
	Conversely, the Weierstrass convergence theorem and the chain rule give
	\begin{align}
	\alpha_r&=\chi_r'(\infty)=\mathop{\tu{lim}}^{}_{n\to\infty}\chi_{r,n}'(\infty)
	=\mathop{\tu{lim}}^{}_{n\to\infty}\eta_{r,n}'(\infty)\cdot(\phi_{r,n}^{-1})'(\infty)\notag\\
	&=\mathop{\tu{lim}}^{}_{n\to\infty}\eta'_{r,n}(\infty)\overset{\eqref{equ:distortion}}{\longrightarrow} 1\tu{\ \ \ as\ \ }r\to0.\notag
	\end{align}
	The proof of the theorem is now complete.
\end{proof}

\emph{Acknowledgment. } The authors would like very much to thank Professors Guizhen Cui and Lei Tan for introductions and many helpful suggestions. The second author
also wants to thank China Scholarship Council for supports.

\end{document}